\documentclass[10pt]{article}
\pdfoutput=1
\usepackage{amsmath, amsfonts, amsthm, amssymb, setspace, geometry, mathrsfs}
 \usepackage[bookmarks=false]{hyperref}
\title{\textbf{On the $ICSPC$-property of finite subgroups}\thanks{\footnotesize \scriptsize\emph{E-mail address:}
      zsmcau@cau.edu.cn\,(S. Zhang).}}
 
\author{Shengmin Zhang\\
\quad
\\
{\small College of Science,
China Agricultural University,
Beijing 100083, China}}
\date{}
\newtheorem{theorem}{Theorem}[section]

\newtheorem{lemma}[theorem]{Lemma}

\theoremstyle{definition}
\newtheorem{definition}[theorem]{Definition}

\allowdisplaybreaks
\begin{document}
\maketitle
\begin{abstract}
Let $G$ be a finite group and $H$ be a subgroup of $G$. Then $H$ is said to be a $p$-$CAP$-subgroup of $G$, if $H$ covers or avoids any $pd$-chief factor of $G$. Furthermore, $H$ is said to be a strong $p$-$CAP$-subgroup of $G$, if for any $H \leq K \leq G$, $H$ is a $p$-$CAP$-subgroup of $K$. A subgroup $L$ is called an $ICSPC$-subgroup of $G$, if $[L,G] \cap L \leq L_{spcG}$, where $L_{spcG}$ denotes a strong $p$-$CAP$-subgroup of $G$ contained in $L$. In this paper, we investigate the structure of $G$ under the assumption that certain subgroups of $G$ are $ICSPC$-subgroups of $G$. Characterizations for $p$-nilpotency and solvably saturated formation are obtained. We also get several criteria for the structure of $G$ under the assumption that certain subgroups of $G$ are $ICSPC$-subgroups of $G$.

\!\!\!\!\!\!\!\!\!\textbf{MSC}:20D10, 20D20
\end{abstract}

\section{Introduction}\label{00001}
All groups considered in this paper will be finite, and $G$ always denotes a finite group. Let $H$ be a subgroup of $G$. Then $H$ is said to satisfy the covering-avoidance property in $G$ (or $H$ is said to be a $CAP$-subgroup of $G$), if for any chief $G$-factor $L/K$, either $L H = K H$ or $L \cap H = K \cap H$. This concept was first introduced by W. Gasch{\"u}tz in {{\cite{GA}}}. Many researchers have investigated the influence of the covering-avoidance property in the structure of $G$ in  \cite{GI,GS,TO}. In {{\cite{BB}}}, A. Ballester-Bolinches {\it et  al.} generalized the basic definition of covering-avoidance property and introduced the concept of strong covering-avoidance property. We list the concept as follows.
\begin{definition}[{{\cite{BB}}}]
A subgroup $H$ of $G$ is called a strong $CAP$-subgroup of $G$, if for any $H \leq K \leq G$, $H$ is a $CAP$-subgroup of $K$. 
\end{definition}
This way of generalization inspires us to make some creations. Another way of generalization of the covering-avoidance property was introduced in {{\cite{FG}}} and was further researched by X. He and Y. Wang in {{\cite{HW}}} and Z. Gao {\it  et al} in {{\cite{GQ}}}. Actually, they gave the following definition.
\begin{definition}[{{\cite{FG}}}]
A subgroup $H$ of a group $G$ is called a $p$-$CAP$-subgroup of $G$ if $H$ covers or avoids every $pd$-chief factor of $G$, where $pd$-chief factor denotes the chief factor of $G$ with order divided by $p$.
\end{definition}
In the paper of X. He and Y. Wang {{\cite{HW}}}, they gave some characterizations for $p$-supersolvability, $p$-nilpotency under the assumption that certain subgroups of $G$ are $p$-$CAP$-subgroups of $G$, sometimes the condition that $G$ is $p$-soluble was needed. In {{\cite{GQ}}}, Z. Gao {\it et al.} gave some equivalent conditions for a group $G$ to be $p$-solvable which are related to $p$-$CAP$-subgroups. 

Now we will introduce a series of concepts which seem to have no relationship with the covering-avoidance property. Now let $H$ be a subgroup of $G$. Then $[H,G]$ denotes the commutator group of $H$ and $G$, and $H^G$ denotes the normal closure of $G$ in $H$, i.e. the smallest normal subgroup of $G$ which contains $H$. It is easy to find that $[H,G]$ is normal in $G$, and $H [H,G] =H^G$. It is natural to to consider the intersection of $H$ and $[H,G]$ since $H^G$ and $[H,G]$ have a strong relationship to the structure of $G$. In fact, we can restrict the intersection of $H$ and $[H,G]$ into certain subgroups of $H$, and get some concepts related to the properties of the intersection. For instance, Y. Gao an X. Li {{\cite{GL}}} introduced the definition as follows.
\begin{definition}[{{\cite{GL}}}]
A subgroup $H$ of $G$ is called an $IC \Phi$-subgroup of $G$, if $H \cap [H,G] \leq \Phi (H)$, where $\Phi (H)$ denotes the Frattini subgroup of $H$.
\end{definition}
Then inspiration of new structure comes to our mind. Instead of letting the restriction of the intersection be the Frattini subgroup, can we change the restriction into a subgroup of $H$ which is related to the $CAP$-subgroups or $p$-$CAP$-subgroups? Actually, it was Y. Gao and X. Li in {{\cite{GL3}}} who first combined the concepts of covering-avoidance property and $IC$-property. In their paper, they introduced the following definition.
\begin{definition}[{{\cite{GL3}}}]
$A$ subgroup $H$ of a group $G$ is called an $ICC$-subgroup of $G$, if $H \cap [H, G] \leq  H_{cG}$, where $H_{cG}$ is a $CAP$-subgroup of $G$ contained in $H$.
\end{definition}[{{\cite{ZSM2}}}]
Now the inspiration immediately comes to our mind. What will happen if we change the $CAP$-subgroup into another covering-avoidance subgroup? In fact, we change the $CAP$-subgroup into $p$-$CAP$-subgroup in {{\cite{ZSM2}}}, and introduced the following definition.
\begin{definition}[{{\cite{ZSM2}}}]
A subgroup $H$ of $G$ is said to be an $ICPC$-subgroup of $G$, if $H \cap [H,G] \leq H_{pcG}$, where $H_{pcG}$ denotes a $p$-$CAP$-subgroup of $G$ contained in $H$.
\end{definition}
In this paper, we obtained the characterization of $p$-nilpotency, solvably saturated formation and other results which are closely related to the structure of $G$ under the assumption that certain subgroups are $ICPC$-subgroups of $G$. Recall that a saturated formation $\mathfrak{F}$ is said to be solvably saturated, if $G \in \mathfrak{F}$ whenever $G / \Phi^{*} (G) \in \mathfrak{F}$. This concept was first introduced by W. Guo and A. N. Skiba {{\cite{GS2}}}.
Now the way to generalize the concept of $IC$-property has been listed above. We just need a generalization of $CAP$-groups or $p$-$CAP$-groups. Actually, we introduced the following definition in {{\cite{ZSM}}}.
\begin{definition}[{{\cite{ZSM}}}]
A subgroup $H$ of a group $G$ is called a strong $p$-$CAP$-subgroup of $G$ if for any subgroup $H \leq K \leq G$, $H$ is a $p$-$CAP$-subgroup of $K$. 
\end{definition}
So we'd like to combine the $IC$-property and strong $p$-$CAP$-subgroups just like what we did in our former paper {{\cite{ZSM2}}}. Hence we introduce the following concept by the inspiration above.
\begin{definition}
A subgroup $H$ of $G$ is said to be an $ICSPC$-subgroup of $G$, if $H \cap [H,G] \leq H_{spcG}$, where $H_{spcG}$ denotes a strong $p$-$CAP$-subgroup of $G$ contained in $H$.
\end{definition}
In this paper, we investigate the basic properties of $ICSPC$-subgroups, and obtained some characterizations for the structure of $G$ under the assumption that certain subgroups of $G$ are $ICSPC$-subgroups of $G$. Since an $ICSPC$-subgroup of $G$ is actually an $ICPC$-subgroup of $G$, some characterizations of solvably saturated formation and  $p$-nilpotency are obtained directly by {{\cite{ZSM2}}}. Our main results are listed as follows.
\begin{theorem}\label{3.3}
Let $N$ be a normal subgroup of $G$ and $P$ a Sylow $p$-subgroup of $N$, where $p \in \pi (N)$. Assume that every cyclic subgroup of $P$ of order $p$ or $4$ (if $P$ is a non-abelian $2$-group) is an
$ICSPC$-subgroup of $G$. Then $N \leq Z_{p\mathfrak{U}} (G)$.
\end{theorem}
\begin{theorem}\label{3.4}
Let $N$ be a normal subgroup of $G$, and $P$ be a Sylow $p$-subgroup of   $N$, where $p \in \pi(N)$. Assume that every maximal subgroup of $P$ is an $ICSPC$-subgroup of $G$, then $N \leq Z_{p\mathfrak{U}} (G)$ or $|N|_p =p$.
\end{theorem}
\section{Preliminaries}
In this section, we investigate the basic properties of $ICSPC$-subgroups, and discover the inheritance of $ICSPC$-subgroups. One can find directly that the inheritance of strong $p$-$CAP$-subgroups is much stronger than almost any other generalizations of covering-avoidance property.
\begin{lemma}\label{0}
Let $H$ be a $p$-CAP subgroup of $G$ and $N \unlhd G$. Then the following statements are true.
\begin{itemize}
\item[(1)] $N$ is a $p$-CAP subgroup of $G$.
\item[(2)] $HN/N$ is a $p$-CAP subgroup of $G/N$.
\end{itemize}
\end{lemma}
\begin{proof}
(1) It follows directly from {{\cite[Lemma 2.1 (1)]{GQ}}}.

(2) Let $H_1 /N$, $H_2/N$ be two normal subgroups of $G/N$ such that $(H_2 /N)/(H_1 /N)$ is a $pd$-chief factor of $G/N$. Clearly $N \leq H_1,H_2$ are two normal subgroups of $G$ with $H_2 \geq H_1$. We conclude from the Correspondence Theorem that $H_2/H_1$ is a chief factor of $G$ with $p\,|\,|(H_2 /N)/(H_1 /N)|=|H_2/H_1|$. Hence $H_2 /H_1$ is a $pd$-chief factor of $G$. Therefore $H_1 H = H_2 H$ or $H_1 \cap H = H_2 \cap H$. The former case suggests that 
$$HN/N \cdot H_1/N = (HN \cdot H_1)/N = H_1 H/N = H_2 H/N = (HN \cdot H_2)/N =HN/N \cdot H_2/N.$$
The latter case suggests from Dedekind Modular Law that 
\begin{align*}
HN/N \cap H_1 /N = (HN \cap H_1)/N = N(H \cap H_1)/N =  N(H \cap H_2)/N= (HN \cap H_2)/N = HN/N \cap H_2 /N.
\end{align*}
Thus $HN/N$ either covers or avoids every $pd$-chief factor of $G/N$. Therefore $HN/N$ is a $p$-$CAP$-subgroup of $G/N$.
\end{proof}
\begin{lemma}\label{1}
Let $H$ be a strong $p$-$CAP$-subgroup of $G$ and $N \unlhd G$. Then the following statements are true.
\begin{itemize}
\item[(1)] $N$ is a strong $p$-$CAP$-subgroup of $G$.
\item[(2)] $H$ is a strong $p$-$CAP$-subgroup for any $K \geq H$.
\item[(3)] $HN/N$ is a strong $p$-$CAP$-subgroup of $G/N$.
\end{itemize}
\end{lemma}
\begin{proof}
(1) It follows from $N \unlhd K$ for any $N \leq K \leq G$ and Lemma \ref{0} that $N$ is a $p$-$CAP$-subgroup for any $N \leq K \leq G$. Hence $N$ is a strong $p$-$CAP$-subgroup of $G$.

(2) Since $H$ is a strong $p$-$CAP$-subgroup of $G$, $H$ is a $p$-$CAP$-subgroup for any $H \leq R \leq G$. Therefore $H$ is a $p$-$CAP$-subgroup for any $H \leq R \leq K$. Hence we conclude that $H$ is a strong $p$-$CAP$-subgroup of $K$.

(3) Since $H$ is a $p$-$CAP$-subgroup for any $HN \leq K \leq G$, it follows from Lemma \ref{0} (2) and $N \unlhd K$ that $HN/N$ is a $p$-$CAP$-subgroup for any $HN/N \leq K/N \leq G/N$. Thus it indicates that $HN/N$ is a strong $p$-$CAP$-subgroup of $G/N$.
\end{proof}
\begin{lemma}\label{2}
Let $G$ be a finite group and $N$ a normal subgroup of $G$. Suppose that $H$ is an $ICSPC$-subgroup of $G$, and $(|H|,|N|) = 1$. Then $H N/N$ is an $ICSPC$-subgroup of $G/N$.
\end{lemma}
\begin{proof}
Since $(|H|,|N|)=1$, it follows that:
$$|HN \cap [H,G]| = \frac{|HN||[H,G]|}{|HN[H,G]|}=\frac{|N \cap H[H,G]||H||[H,G]|}{|H[H,G]|}=|H \cap [H,G]||N \cap H[H,G]|. $$
Hence we have that $HN \cap [H,G] = (H \cap [H, G])(N \cap H[H, G])$, and:
$$H N \cap [H N, G]N = H N \cap [H, G]N = (H N \cap [H, G])N = (H \cap [H, G])N. $$
Thus we conclude that:
$$H N/N \cap [H N/N, G/N] = (H N \cap [H N, G]N)/N = (H \cap [H, G])N/N \leq  H_{spcG} N/N. $$
It yields from lemma \ref{1} (3) that $H_{spcG} N/N$ is a strong $p$-$CAP$-subgroup of $G/N$ which is contained in $HN/N$. Therefore we get that $H N/N \cap [H N/N, G/N] \leq (HN/N)_{spc(G/N)}$. In other words, $H N/N$ is an $ICSPC$-subgroup of $G/N$.
\end{proof}
\begin{lemma}[{{\cite[Theorem 2.1.6]{BE}}}]\label{3}
Let $G$ be a $p$-supersolvable group. Then the derived subgroup $G'$
of $G$ is $p$-nilpotent. In particular, if $O_{p'} (G)=1$, then $G$ has a unique Sylow $p$-subgroup.
\end{lemma}
\begin{lemma}[{{\cite[Lemma 2.8]{SL}}}]\label{5}
Let $U$, $V$ and $W$ be subgroups of $G$. Then the following statements are equivalent:
\begin{itemize}
\item[(1)] $U \cap V W = (U \cap V)(U \cap W)$;
\item[(2)] $U V \cap U W = U(V \cap W)$.
\end{itemize}
\end{lemma}
\section{Proofs of the main theorems}
As we have said in section \ref{00001}, some characterizations of solvably saturated formation and  $p$-nilpotency are obtained directly by {{\cite{ZSM2}}} since an $ICSPC$-subgroup of $G$ is actually an $ICPC$-subgroup of $G$. Now we list several directly corollaries from {{\cite{ZSM2}}}.
\begin{theorem}\label{3.1}
Let $G$ be a finite group and $P$ be a normal $p$-group of $G$. Suppose that any cyclic subgroup of $P$ of order $p$ and $4$ (if $P$ is a non-abelian $2$-group) either is an $ICSPC$-subgroup of $G$ or has a supersolvable supplement in $G$. Then $P \leq Z_{\mathfrak{U}} (G)$. 
\end{theorem}
\begin{proof}
It follows directly from {{\cite[Theorem 1.6]{ZSM2}}}.
\end{proof}
\begin{theorem}\label{3.01}
Let $\mathfrak{F}$ be a solvably saturated formation containing $\mathfrak{U}$ and $N$ a solvable normal subgroup of $G$ such that $G/N \in \mathfrak{F}$. Assume that every cyclic subgroup of every non-cyclic Sylow subgroup $P$ of $F(N)$ of prime order or $4$ (if $P$ is a non-abelian $2$-group) either is an $ICSPC$-subgroup of $G$ or has a supersolvable supplement in $G$. Then $G \in \mathfrak{F}$.
\end{theorem}
\begin{proof}
It follows directly from {{\cite[Corollary 1.7]{ZSM2}}}.
\end{proof}
\begin{theorem}\label{3.2}
Let $G$ be a finite group, $E$ be a normal subgroup of $G$, and $P$ be a Sylow $p$-subgroup of $E$, where $p = \min \pi(E)$. Assume that every cyclic subgroup of $P$ of order $p$ and $4$ (if $P$ is a non-abelian $2$-group) is an $ICSPC$-subgroup of $G$, then $E$ is $p$-nilpotent.
\end{theorem}
\begin{proof}
It follows directly from {{\cite[Theorem 1.8]{ZSM2}}}.
\end{proof}
With the useful theorems as necessary lemmas, we are now ready to prove our main theorems.
\begin{proof}[Proof of Theorem \ref{3.3}]
Suppose that the theorem is not true, and let $(G,N)$ be a pair of  counterexample with $|G| +|N|$ minimal.
\begin{itemize}
\item[\textbf{Step 1.}] $O_{p'} (N) =1$.
\end{itemize}
Assume that $O_{p'} (N) >1$. Then clearly $N / O_{p'} (N)$ is a normal subgroup of $G/O_{p'} (N)$, $P O_{p'} (N)/O_{p'} (N)$ is a Sylow $p$-subgroup of $G/O_{p'} (N)$. Since $(|O_{p'} (N)|,p)=1$, every cyclic subgroup of $PO_{p'} (N)/O_{p'} (N)$ of order $p$ or $4$ (if $P$ is a non-abelian $2$-group) can be written as $\langle x \rangle O_{p'} (N)/O_{p'} (N)$, where $x \in P$, and $o(x) =p$ or $4$ (if $P$ is a non-abelian $2$-group). By our hypothesis, $\langle x \rangle$ is an
$ICSPC$-subgroup of $G$. By Lemma \ref{2}, $\langle x \rangle O_{p'} (N)/O_{p'} (N)$ is an $ICSPC$-subgroup of $G/O_{p'} (N)$. Hence the pair $(G/O_{p'} (N),PO_{p'} (N)/O_{p'} (N))$ satisfies the hypothesis of the theorem. Therefore $N /O_{p'} (N) \leq Z_{p\mathfrak{U}} (G/O_{p'} (N)) = Z_{p\mathfrak{U}} (G)/O_{p'} (N)$ by the choice of $(G,N)$. Hence $N \leq Z_{p\mathfrak{U}} (G)$, a contradiction and so $O_{p'} (N) =1$.
\begin{itemize}
\item[\textbf{Step 2.}] $N=G$, and $O_{p'} (G) =1$.
\end{itemize}
Suppose that $N <G$. By lemma \ref{1} (2), every cyclic subgroup of $P$ of order $p$ or $4$ (if $P$ is a non-abelian $2$-group) is an
$ICSPC$-subgroup of $N$. Hence the pair $(N,N)$ satisfies the hypothesis, therefore $N \leq Z_{p\mathfrak{U}} (N)$, i.e. $N$ is $p$-supersolvable by the minimal choice of $(G,N)$. Then by Lemma \ref{3}, we conclude that $P \unlhd G$. Hence $P$ satisfies the hypothesis of Theorem \ref{3.1}. It follows that $P \leq Z_{\mathfrak{U}} (G) \leq Z_{p\mathfrak{U}} (G)$, and so $N \leq Z_{p\mathfrak{U}} (G)$, a contradiction. Thus we have $N=G$ and $O_{p'} (G)=1$.
\begin{itemize}
\item[\textbf{Step 3.}] $Z_{p\mathfrak{U}} (G)$ is the unique normal subgroup of $G$ such that $G/Z_{p\mathfrak{U}} (G)$ is a chief factor of $G$. Moreover, $O_p (G) =Z(G) =Z_{\mathfrak{U}} (G)$ is the unique  Sylow $p$-subgroup of  $Z_{p\mathfrak{U}} (G)$.
\end{itemize}
Now let $G/K$ be a chief factor of $G$. It is clear that the pair $(G,K)$ satisfies the hypothesis of the theorem. Therefore $K \leq Z_{p\mathfrak{U}} (G)$ by the minimal choice of $(G,N)$. Hence $K = Z_{p\mathfrak{U}} (G)$ since $G/K$ is a chief factor of $G$. Therefore $Z_{p\mathfrak{U}} (G)$ is the unique normal subgroup of $G$ such that $G/Z_{p\mathfrak{U}} (G)$ is a chief factor of $G$. Since $O_p (G)$ is a normal $p$-subgroup of $G$, it is clear that $O_p (G)$ satisfies the hypothesis of Theorem \ref{3.1}. Therefore we have $O_p (G) \leq Z_{\mathfrak{U}} (G) \leq Z_{p\mathfrak{U}} (G)$. Since $Z_{p\mathfrak{U}} (G)$ is $p$-supersolvable, $O_{p'} (Z_{p\mathfrak{U}} (G))=1$ as $Z_{p\mathfrak{U}} (G) \unlhd G$ and $O_{p'} (G) =1$, we conclude from Lemma \ref{3} that $ Z_{p\mathfrak{U}} (G)$ has a unique Sylow $p$-subgroup. Clearly this Sylow $p$-subgroup is $O_p (G)$. Suppose that $G^{\mathfrak{U}} <G$, then $G^{\mathfrak{U}} \leq Z_{p\mathfrak{U}} (G)$, and $G^{\mathfrak{U}} \cap O_p (G)$ is the unique Sylow $p$-subgroup of $G^{\mathfrak{U}}$. Let $P_0 = G^{\mathfrak{U}} \cap O_p (G)$. Since $(G/P_0)/(G^{\mathfrak{U}}/P_0)$ is supersolvable and $G^{\mathfrak{U}} /P_0$ is a $p'$-group, we get that $G/P_0$ is $p$-supersolvable. Since $P_0 \leq O_p (G) \leq  Z_{\mathfrak{U}} (G)$,
 we conclude that $G$ is $p$-supersolvable, a contradiction. Hence we have that $G^{\mathfrak{U}} =G$. By {{\cite[Chapter IV, Theorem 6.10]{DH}}}, we get that $[G^{\mathfrak{U}}, Z_{\mathfrak{U}} (G)]=1$. Thus we have that $Z_{\mathfrak{U}} (G) \leq Z(G)$. It follows from $O_{p'} (Z(G)) \leq O_{p'} (G)=1$ that $Z(G) \leq O_p (G)$. Hence it follows that $O_p (G) =Z(G) =Z_{\mathfrak{U}} (G)$.
 \begin{itemize}
\item[\textbf{Step 4.}] Final contradiction.
\end{itemize}
Suppose that $G'<G$ or $G^{\mathfrak{U}}<G$ holds. Then it indicates from (3) that $G'\leq Z_{p\mathfrak{U}} (G)$ or $G^{\mathfrak{U}}\leq Z_{p\mathfrak{U}} (G)$. This implies that $G \leq Z_{p\mathfrak{U}} (G)$, a contradiction.Thus we conclude that $G = G' = G^{\mathfrak{U}}$. Now we assume that $P$ is abelian. Then it follows from {{\cite[Chapter  VI, Theorem 14.3]{H1}}} that $Z(G)=1$. Hence by (2) and (3), we get that $Z_{p\mathfrak{U}} (G)=1$. This implies that $G$ is a simple group. Now let $x \in P$ of order $p$. If $\langle x \rangle$ is a strong $p$-$CAP$-subgroup of $G$, it follows from the simplicity of $G$ that $\langle x \rangle$ covers or avoids $G/1$. The former case suggests that $G = \langle x \rangle$, a contradiction to the fact that $G = G'$. The latter case suggests that $1 = \langle x \rangle$, a contradiction as well. Hence we have that $\langle x \rangle$ is not a strong $p$-$CAP$-subgroup of $G$, and so $\langle x \rangle_{spcG}<\langle x \rangle$, i.e. $\langle x \rangle_{spcG}\leq  \Phi(\langle x \rangle)$. By our hypothesis, $\langle x \rangle$ is an $ICSPC$-subgroup of $G$. Thus $\langle x \rangle \cap [\langle x \rangle,G] \leq \Phi(\langle x \rangle)$. Suppose that $[\langle x \rangle,G]=1$, then $\langle x \rangle \leq Z(G)$, a contradiction. Assume that $[\langle x \rangle,G]>1$, it follows again from the simplicity of $G$ that $[\langle x \rangle,G]=G$. Then $\langle x \rangle \cap [\langle x \rangle,G] = \langle x \rangle \leq \Phi(\langle x \rangle)$, a contradiction as well. Therefore $P$ is not abelian. By {{\cite[Chapter  VI, Theorem 5.5]{H1}}}, there exists $x \in P$ of order $p$ or $4$ such that $x \notin Z(G) = O_p (G)$. Then it is obvious that $O_p (G) < P$, hence $p\,|\,|G/Z_{p\mathfrak{U}} (G)|$. Suppose again that $\langle x \rangle$ is a strong $p$-$CAP$-subgroup of $G$. Then $\langle x \rangle$ covers or avoids $G/Z_{p\mathfrak{U}} (G)$. The former case indicates that $G = \langle x \rangle Z_{p\mathfrak{U}} (G)$, i.e. $G/Z_{p\mathfrak{U}} (G) \leq \langle x \rangle$. Hence $G$ is $p$-supersolvable, a contradiction. The latter case suggests that $\langle x \rangle \leq Z_{p\mathfrak{U}} (G)$, which yields that $x \in O_p (G) = Z(G) $, a contradiction as well. Therefore we get that $\langle x \rangle_{spcG}<\langle x \rangle$, i.e. $\langle x \rangle_{spcG}\leq  \Phi(\langle x \rangle)$. It is easy to find that $\langle x \rangle \not\leq Z_{p\mathfrak{U}} (G)$. Then it indicates from $G/Z_{p\mathfrak{U}} (G)$ is a chief factor of $G$ that $G = \langle x \rangle [ \langle x \rangle,G] Z_{p\mathfrak{U}} (G)$. It is easy to see that $[\langle x \rangle,G] \cap P \neq 1$. Suppose that $[\langle x \rangle,G]<G$. Then every cyclic subgroup of $[\langle x \rangle,G] \cap P$ of order $p$ or $4$ (if $[\langle x \rangle,G] \cap P$ is non-abelian) is an $ICSPC$-subgroup of $G$. Hence the pair $(G,[\langle x \rangle,G])$ satisfies our hypothesis, and so $[\langle x \rangle,G] \leq Z_{p\mathfrak{U}} (G)$ by the minimal choice of $(G,N)$. Thus it follows that $\langle x \rangle  Z_{p\mathfrak{U}} (G) = G$. Therefore we conclude that $G \leq Z_{p\mathfrak{U}} (G)$, a contradiction. Hence we get that $[\langle x \rangle,G]=G$, which asserts that $\langle x \rangle \cap [\langle x \rangle,G] = \langle x \rangle \leq \Phi(\langle x \rangle)$, a contradiction as well and the proof is complete.
\end{proof}
\begin{proof}[Proof of Theorem \ref{3.4}]
Suppose that the theorem is false, and let $(G,N)$ be a counterexample  with $|G|+|N|$ minimal. 
\begin{itemize}
\item[\textbf{Step 1.}] $O_{p'} (N)=1$, and $N=G$.
\end{itemize}
Assume that $O_{p'} (N) \neq 1$, it follows from Lemma \ref{1} (3) that  the pair $(G/O_{p'} (N), N/O_{p'} (N))$ satisfies the hypothesis of the theorem. Therefore $N /O_{p'} (N) \leq Z_{p\mathfrak{U}} (G/O_{p'} (N))=Z_{p\mathfrak{U}} (G)/O_{p'} (N) $ by the choice of $(G,N)$. Thus $N \leq Z_{p\mathfrak{U}} (G)$, a contradiction. Now assume that $N<G$. It indicates from Lemma \ref{1} (2) that the pair $(N,N)$ satisfies the hypothesis of the theorem. Hence $N$ is $p$-supersolvable by the choice of $(G,N)$. By Lemma \ref{3}, we get that $P \unlhd G$. Therefore $P$ satisfies the hypothesis of Theorem \ref{3.1} and so $P \leq Z_{\mathfrak{U}} (G) \leq Z_{p\mathfrak{U}} (G)$, which implies that $N \leq Z_{p\mathfrak{U}} (G)$, a contradiction. Thus we have $N=G$ and $O_{p'} (G)=1$. 
\begin{itemize}
\item[\textbf{Step 2.}] Let $K$ be a minimal normal subgroup of $G$. Then $G/K$ is $p$-supersolvable or $|G/K|_p =p$.
\end{itemize}
Now let $K$ be a minimal normal subgroup of $G$, and $M/K$ be a maximal subgroup of $PK/K$. Let $P_1 := M \cap P$. Then $M = M \cap PK = K (M \cap P) = P_1 K$, and $P \cap K = M \cap P \cap K = P_1 \cap K$. Hence we get that $|P:P_1| = |PK:M| = p$, i.e. $P_1$ is a maximal subgroup of $P$. Clearly $P_1 \cap K = P \cap K$ is a Sylow $p$-subgroup of $K$, then we have
\begin{align*}
|K \cap P_1 [P_1,G]|_p = |K|_p = |K \cap P_1| = |(K \cap P_1)(K \cap [P_1,G])|_p,
\end{align*}
and
\begin{align*}
|K \cap P_1 [P_1,G]|_{p'} = \frac{|K|_{p'}|P_1 [P_1,G]|_{p'}}{|K P_1 [P_1,G]|_{p'}} =\frac{|K|_{p'}|[P_1,G]|_{p'}}{|K[P_1,G]|_{p'}} =|K \cap  [P_1,G]|_{p'} =|(K \cap P_1)(K \cap [P_1,G])|_{p'}.
\end{align*}
Therefore we conclude that $K \cap P_1 [P_1,G] = (K \cap P_1)(K \cap [P_1,G])$. By Lemma \ref{5}, we assert that 
$$P_1 K \cap [P_1,G]K = (P_1 \cap [P_1,G])K. $$
Hence it follows that 
$$P_1 K / K \cap [P_1 K/K,G/K] = (P_1 K \cap [P_1 K,G])/K \leq (P_1 K \cap [P_1,G]K)/K = (P_1 \cap [P_1,G])K/K. $$
Since $P_1$ is an $ICSPC$-subgroup of $G$, there exists a strong $p$-$CAP$-subgroup $(P_1)_{spcG}$ of $G$ contained in $P_1$ such that $P_1 \cap [P_1, G] \leq (P_1)_{spcG}$. By Lemma \ref{1} (3), it yields that $(P_1)_{spcG} K/K$ is a strong $p$-$CAP$-subgroup of $G/K$. Therefore we have 
$$P_1 K / K \cap [P_1 K/K,G/K] \leq (P_1 K /K)_{spc(G/K)}.  $$
In other words, $P_1 K /K = M/K$ is an $ICSPC$-subgroup of $G/K$. By the randomness of $M/K$, we get that the pair $(G/K, PK/K)$ satisfies the hypothesis. Hence by the choice of $(G,N)$, it follows that $G/K \leq  Z_{p\mathfrak{U}} (G/K)$ or $|G/K|_p =p$, i.e. $G/K$ is $p$-supersolvable or $|G/K|_p =p$.
\begin{itemize}
\item[\textbf{Step 3.}] $K$ is the unique minimal normal subgroup of $G$.
\end{itemize}
Let $K_1$ and $K_2$ be two different minimal normal subgroups of $G$. Then it indicates from step 2 that $G/K_i$ is $p$-supersolvable or $|G/K_i|_p =p$ for $i=1,2$. Since $O_{p'} (G)=1$, it yields that $p\,|\,|K_i|$, $i=1,2$. Assume firstly that both $G/K_1$ and $G/K_2$ are $p$-supersolvable, then $G/(K_1 \cap K_2) \cong G$ is $p$-supersolvable, a contradiction. Now without loss of generality, we may assume that $G/K_1$ is $p$-supersolvable and $|G/K_2|_p =p$. Since $K_1 K_2 / K_1$ is a minimal normal subgroup of $G/K_1$ and $p\,|\,|K_2|$, it indicates from the $p$-supersolvability  of $G/K_1$ that $|K_2| = |K_1 K_2/K_1|=p$, and so $|P| = p^2$ since $|G/K_2|_p =p$. Also, we get from $K_1 \cap K_2 =1$ and $p\,|\,|K_1|$ that $|K_1|_p = p$. Now let $E = K_1 \cap P$, then clearly $|E|=p$, and $E$ is a maximal subgroup of $P$. Assume that $E_G \neq 1$ or $E_{spcG} \neq 1$, then $E_G=E$ or $E_{spcG} =E$. The first case indicates that $E$ is normal in $G$, hence $E = K_1$. It follows from $G/K_1$ is $p$-supersolvable that $G$ is $p$-supersolvable, a contradiction. The second case suggests that $E$ covers or avoids $K_1/1$ since $p\,|\,|K_1|$. Therefore $E K_1 = E$ or $E \cap K_1 =1$. The former one implies that $K_1 =E$, a clearly contradiction. The latter one suggests that $E=1$, a contradiction as well. Thus we get that $E_G = E_{spcG} =1$. Note that $E$ is a maximal subgroup of $P$, we conclude that $E$ satisfies the hypothesis and so $E \cap [E,G]=1$. Since $E^G = E[E,G] =K_1 >[E,G]$, we get that $[E,G]=1$ and so $E \leq Z(G)$, which implies that $E = K_1$, a contradiction. Finally, suppose that $|G/K_i|_p =p$, $i=1,2$. Then $|PK_i/K_i| =p = |P/(P \cap K_i)|$, $i=1,2$. It follows  from $(P \cap K_1 ) \cap (P \cap K_2)=1$ that $|P \cap K_1|=|P \cap K_2|=p$, and so $|P| = p^2$. Again, let $E_1 = P \cap K_1$. Suppose that $(E_1)_G \neq 1$ or $(E_1)_{spcG} \neq 1$, then $(E_1)_G = E_1$ or $(E_1)_{spcG} =E_1$. The first case indicates that $E_1 =K_1$. Then $|K_1|=p$, and so $K_1 K_2/K_2$ is a minimal normal subgroup of $G/K_2$ with order $p$. It follows from $|G/K_2|_p=p$ that $G/K_2$ is $p$-supersolvable. By similar argument, we derive a contradiction. The second case suggests that $K_1 =E_1$ or $E=1$. If $K_1 =E_1$, by similar argument as above, we conclude that 
 $G/K_2$ is $p$-supersolvable, and so a contradiction as well. Therefore $(E_1)_G = (E_1)_{spcG}=1$. Note that $E_1$ is a maximal subgroup of $P$, we conclude that $E_1$ satisfies the hypothesis and so $E_1 \cap [E_1,G]=1$. Since $(E_1)^G = E_1 [E_1 ,G] =K_1 >[E,G]$, we get that $[E_1,G]=1$ and so $E_1 \leq Z(G)$, which implies that $E_1 = K_1$. A similar argument yields that $G/K_2$ is $p$-supersolvable, a contradiction. Hence there is a unique minimal normal subgroup of $G$ and step 3 holds.
\begin{itemize}
\item[\textbf{Step 4.}] $K \not\leq \Phi (P)$.
\end{itemize}
Assume that $K \leq \Phi(P)$. Then $K \leq \Phi (G)$. Then we conclude from step 2 that $G/K$ is $p$-supersolvable or $|G/K|_p =p$. Suppose that the former case holds, then $G$ is $p$-supersolvable, a contradiction. Hence we get that $|G/K|_p =p$, and so $|P/K|=p$. This indicates that $P$ is cyclic, and $|K|=p$, $|P|=p^2$ since $K$ is the minimal normal subgroup of $G$. Now let $A/K = O_{p'} (G/K)$. Then we get that $A \cap P \leq K \leq \Phi (P)$. By {{\cite[Chapter IV, Theorem 4.7]{H1}}}, we have that $A$ is $p$-nilpotent. It follows from $O_{p'} (G)=1$ that $A$ is a $p$-group, hence $A=K$ and so $O_{p'} (G/K)=1$. Let $B/K$ be a minimal normal subgroup of $G/K$. Then it indicates from $O_{p'} (G/K)=1$ that $|B/K|_p=p$. Therefore $P \leq B$. Then the pair $(G,B)$ satisfies the hypothesis of the theorem. Suppose that $B<G$, then $B \leq Z_{p\mathfrak{U}} (G)$ or $|B|_p =p$ by the minimal choice of $(G,N)$. Since $|B|_p=p^2$, we  assert that $B \leq Z_{p\mathfrak{U}} (G)$. It is obvious that $G/B$ is a $p'$-group, hence $G$ is $p$-supersolvable, a contradiction. Thus $B=G$ and so $G/K$ is a non-abelian simple group. Since $|K|=p$, it yields that $G/C_G (K) \lesssim {\rm Aut} (K)$ is abelian by N-C Theorem. Note that $G/C_G (K) \cong (G/K)/(C_G (K)/K)$, it follows from the simplicity of $G/K$ that $G =C_G(K)$ and so $K \leq Z(G)$. By {{\cite[Chapter VI, Theorem 14.3]{H1}}}, we get that $1 <K \leq G' \cap P \cap Z(G)$, a contradiction. Thus step 4 holds.
\begin{itemize}
\item[\textbf{Step 5.}] Final contradiction.
\end{itemize}
Suppose that $O_p (G)=1$. Let $P_1$ be a maximal subgroup of $P$ containing $K_p = P \cap K$. By our hypothesis, we have $P_1 \cap [P_1,G] \leq (P_1)_{spcG}$. Therefore we get that $(P_1)_{spcG}$ covers or avoids $K/1$, which implies that $(P_1)_{spcG} \geq K$ or $K \cap (P_1)_{spcG}=1$. The former one indicates that $K$ is a $p$-group, a contradiction. Hence we conclude that $K \cap (P_1)_{spcG}=1$. Suppose that $[P_1,G]=1$, then $P_1 \unlhd G$, impossible. It follows directly that $K \leq [P_1,G]$, and so $P \cap K \leq P_1 \cap [P_1 , G] \leq  (P_1)_{spcG}$. Thus we get that $P \cap K = P \cap K \cap (P_1)_{spcG} = P \cap 1 =1$, a contradiction to the fact that $p\,|\,|K|$. Thus $O_p (G)>1$, and so $K \leq O_p (G)$ by step 3. From step 2, we have that $G/K$ is $p$-supersolvable or $|G/K|_p=p$. If the first case holds, then we have $\Phi (G)=1$. Hence $G =K \rtimes M$, where $M$ is a maximal subgroup of $G$, and $P = KM_p$, where $M_p = P \cap M$. Let $P_1$ be a maximal subgroup of $P$ containing $M_p$, then we get $P = K P_1$. Since $P_1$ satisfies our hypothesis, $P_1 \cap [P_1, G] \leq (P_1)_{spcG}$. It follows that $(P_1)_{spcG}$ covers or avoids $K/1$, which implies that $(P_1)_{spcG} \geq K$ or $K \cap (P_1)_{spcG}=1$. The first case suggests that $P_1 \geq (P_1)_{spcG} \geq K$, and so $P = P_1$, a contradiction. Now we claim that $[P_1, G] >1$. Assume that $[P_1, G]=1$, then $P_1 \leq Z(G)$. Suppose that $P_1 \cap K=1$, then $K$ is not the unique minimal normal subgroup, impossible. If $P_1 \cap K>1$, then $K \leq P_1$ and so $P=P_1$,  a contradiction hence $[P_1, G] >1$, i.e. $[P_1, G] \geq K$. Therefore the second case shows that $P_1 \cap K \leq P_1 \cap [P_1, G] \leq (P_1)_{spcG}$, which yields that $P_1 \cap K = P_1 \cap K \cap (P_1)_{spcG} =1$. By $P =KP_1$, we conclude that $|K|=p$, which indicates that $G$ is $p$-supersolvable, impossible. Hence $|G/K|_p=p$, and we may assume that $G/K$ is not $p$-supersolvable. By step 4, there exists a maximal subgroup $P_2$ such that $P = KP_2$. Using the similar argument as above, we conclude that $P_2 \cap K = 1$ and $|K|=p$, which implies that $|P|=p^2$. Now assume that every maximal subgroup $P_0$ of $P$ which is not equal to $K$ has the property that $P_0 \leq [P_0,G]$. Then every maximal subgroup $P_0$ of $P$ which is not equal to $K$ is a strong $p$-$CAP$-subgroup of $G$ since $P_0 \cap [P_0,G] = P_0 \leq (P_0)_{spcG}$. As $|P|=p^2$ and $P=P_2 K$ with $|P_2|=|K|=p$, it  indicates that $P$ is not cyclic. By the fact that $K$ is normal in $G$, we conclude from Lemma \ref{1} (1) that every cyclic subgroup of $P$ of order $p$ or $4$ is a strong $p$-$CAP$-subgroup of $G$. Therefore we conclude that $(G,G)$ satisfies the hypothesis of  Theorem \ref{3.3}. Then $G \leq Z_{p\mathfrak{U}} (G)$, a contradiction. Hence there exists a maximal subgroup $P_3$ of $P$ which is not equal to $K$ such that $P_3 \cap [P_3,G]=1$. Therefore we assert that $[P_3,G] <G$. Now suppose that $(P_3)^G = G$. Then $P_3 [P_3,G]=G$, and so $|G/[P_3,G]|=p$. Since $K \leq [P_3,G]$, we get that $[P_3,G]/K$ is a $p'$-group, which implies that $G/K$ is $p$-supersolvable, a contradiction. Therefore we have $P \leq (P_3) ^G <G$. It is easy to find that the pair $(G,(P_3)^G)$ satisfies the hypothesis of the theorem, hence it indicates that $(P_3)^G$ is $p$-supersolvable or $|(P_3)^G|_p = p$. The latter case yields that $|P|=p$, a contradiction. Thus $(P_3)^G$ is $p$-supersolvable. Since $G/(P_3)^G$ is a $p'$-group, we conclude that $G$ is $p$-solvable and so $G/K$ is $p$-solvable. This implies that $G/K$ is $p$-supersolvable since $|G/K|_p=p$, a contradiction. Hence the proof in complete.
\end{proof}
\!\!\!\!\!\!\!\!\!\textbf{Declarations of interest:} none.

\end{document}